\documentclass[12pt, reqno]{amsart}
\usepackage{amsmath, amsthm, amscd, amsfonts, amssymb, graphicx, color}
\usepackage[bookmarksnumbered, colorlinks, plainpages]{hyperref}

\makeatletter
\@namedef{subjclassname@2020}{%
  \textup{2020} Mathematics Subject Classification}
\makeatother

\newtheorem{theorem}{Theorem}[section]

\theoremstyle{definition}

\theoremstyle{remark}
\newtheorem{remark}[theorem]{Remark}
\numberwithin{equation}{section}

\begin{document}


\title[Local and 2-local automorphisms of nilpotent associative algebras]{Local and 2-local automorphisms of finite-dimensional nilpotent associative algebras}

	\author{F.N.Arzikulov}
\address[F.N.Arzikulov]	{Institute of Mathematics Uzbekistan Academy of Sciences, 4, University street, Olmazor, Tashkent, 100174, Uzbekistan
\newline
		and\newline
	Department of Mathematics, Andijan State University, 129, Universitet Street, Andijan, 170100, Uzbekistan}
	\email{arzikulovfn@rambler.ru}
	
	\author{I.A.Karimjanov}
\address[Iqboljon Karimjanov]{Saint-Petersburg State University, Saint-Petersburg, Russia,
\newline and\newline
Department of Mathematics, Andijan State University, 129, Universitet Street, Andijan,
170100, Uzbekistan,
\newline
		and\newline
		V.I.Romanovskiy Institute of Mathematics, Uzbekistan Academy of Sciences, Univesity Street, 9, Olmazor district, Tashkent, 100174, Uzbekistan}
	\email{iqboli@gmail.com}
	
	\author{S.M.Umrzaqov}
	\address[S.M.Umrzaqov]{	
		Department of Mathematics, Andijan State University, 129, Universitet Street, Andijan,
170100, Uzbekistan }
	\email{sardor.umrzaqov1986@gmail.com}

\thanks{This work was supported by RSF 22-71-10001}	
	\maketitle
	\maketitle
\maketitle

\begin{abstract}
In the present paper automorphisms, local and 2-local automorphisms of $n$-dimensional null-filiform and filiform   associative algebras are studied. Namely, a common form of the matrix of automorphisms and local automorphisms of these algebras is clarified. It turns out that the common form of the matrix of an automorphism on these algebras does not coincide with the local automorphism's matrices common form on these algebras.
Therefore, these associative algebras have local automorphisms that are not automorphisms.
Also, that each 2-local automorphism of these algebras is an automorphism is proved.
\end{abstract}

\medskip

\section*{Introduction}

The Gleason-Kahane-\.{Z}elazko theorem \cite{AMG,JPK_WZ}, which is a fundamental contribution to the theory of Banach algebras,
asserts that every unital linear local homomorphism from an unital complex Banach algebra $A$ into ${\Bbb C}$ is
multiplicative. We recall that a linear map $T$ from a Banach algebra $A$ into a Banach algebra $B$ is said to
be a local homomorphism if for every $a$ in $A$ there exists a homomorphism $\Phi_a : A\to B$, depending on $a$,
such that $T(a)=\Phi_a(a)$.

Later, in \cite{Kad}, R. Kadison introduces the concept of local derivations and proves
that each continuous local derivation from a von Neumann algebra into its dual Banach
bemodule is a derivation. B. Jonson \cite{Jon} extends the above result by proving that every
local derivation from a C*-algebra into its Banach bimodule is a derivation. In particular, Johnson
gives an automatic continuity result by proving that local derivations of a C*-algebra $A$ into a
Banach $A$-bimodule $X$ are continuous even if not assumed a priori to be so
(cf. \cite[Theorem 7.5]{Jon}). Based on these results, many authors have studied
local derivations on operator algebras.

A similar notion, which characterizes non-linear generalizations of automorphisms, was introduced
by  \v{S}emrl in \cite{S} as $2$-local automorphisms.
He described such maps on the algebra $B(H)$ of all bounded linear operators on an infinite dimensional separable Hilbert space $H$.

The first results concerning local derivations and automorphisms on finite-dimensional Lie algebras were obtained in \cite{AK}. Namely, in \cite{AK} the authors have proved that every local derivation on semi-simple Lie algebras is a derivation and gave examples of nilpotent finite-dimensional Lie algebras with local derivations which are not derivations.
Sh.A.Ayupov, K.K.Kudaybergenov, B.A.Omirov proved similar results concerning local derivations and automorphisms on simple Leibniz algebras in their recent paper \cite{AKO}. Local automorphisms of certain finite-dimensional simple Lie and Leibniz algebras are investigated in \cite{AKConfe}. Concerning local automorphism, T.Becker, J.Escobar, C.Salas, and R.Turdibaev in \cite{BEST} established that the set of local automorphisms $LAut(sl_2)$  coincides with the group $Aut^{\pm}(sl_2)$ of all automorphisms and anti-automorphisms. Later in \cite{Costantini}  M.Costantini proved that a linear map on a simple Lie algebra is a local automorphism if and only if it is either an automorphism or an anti-automorphism. The local derivation of semisimple Leibniz algebras investigated on \cite{KKY}. Similar results concerning local derivations and automorphisms on Lie superalgebras were obtained in \cite{ChWD,WCN1} and \cite{WCN2}.

In the paper \cite{Ayupov6}, local derivations of solvable Lie algebras are studied, and it is proved that in the class of solvable Lie algebras, there exist algebras that admit local derivations which are not derivation. Also, algebras, every local derivation of which is a derivation, are found. Moreover, every local derivation on a finite-dimensional solvable Lie algebra with model nilradical and the maximal dimension of complementary space is a derivation. Sh.A.Ayupov, A.Kh.Khudoyberdiyev, and B.B.Yusupov proved similar results concerning local derivations on solvable Leibniz algebras in their recent papers \cite{AyuKudYus, Yus}. F.N.Arzikulov, I.A.Karimjanov, and S.M.Umrzaqov established that every local and 2-local automorphisms on the solvable Leibniz algebras with null-filiform and naturally graded non-Lie filiform nilradicals, whose dimension of complementary space is maximal, is an automorphism \cite{AKU}. Recently, local derivations and automorphisms of Cayley algebras,  local derivations on the simple Malcev algebra  and local and 2-local derivations of simple $n-$ary algebras considered in \cite{AK33, FKK,KKY}

In the paper \cite{KUY}, I.A.Karimjanov, S.M.Umrzaqov, and B.Yusupov describe automorphisms, local and 2-local automorphisms of solvable Leibniz algebras with a model or abelian null-radicals. They show that any local automorphisms on solvable Leibniz algebras with a model nilradical, the dimension of the complementary space of which is maximal, is an automorphism.
But solvable Leibniz algebras with an abelian nilradical with a  $1$-dimensional complementary space admit local automorphisms which are not automorphisms.

In the present paper automorphisms, local and 2-local automorphisms of $n$-dimensional filiform and null-filiform associative algebras are studied. Namely, a common form of the matrix of automorphisms and local automorphisms of these algebras is clarified. It turns out that the common form of the matrix of an automorphism on these algebras does not coincide with the local automorphism's matrix's common form on these algebras.
Therefore, these associative algebras have local automorphisms that are not automorphisms.
Also, that each 2-local automorphism of these algebras is an automorphism is proved.

\section{Preliminaries}\label{www}

{\bf Null-filiform and filiform associative algebras.} For an algebra $\mathcal{A}$ of an arbitrary variety, we consider the series
\[\mathcal{A}^1,\quad \quad     \mathcal{A}^{i+1}=\sum\limits_{k=1}^i\mathcal{A}^k
\mathcal{A}^{i-k+1},  \quad \quad i\geq1.\]

We say that an algebra $\mathcal{A}$ is nilpotent if $\mathcal{A}^i=0$ for some $i\in \mathbb{N}$. The smallest integer satisfying $\mathcal{A}^i=0$ is called the index of nilpotency or nilindex of $\mathcal{A}$.

{\bf Definition.} An $n-$dimensional algebra $\mathcal{A}$ is called null-filiform $dim \mathcal{A}^i=(n+1)-i, 1\leq i\leq n+1$.

\begin{theorem} [{\bf \cite{MasOmi}}] \label{2.1}
An arbitrary $n$-dimensional null-filiform associative algebra is isomorphic to the following algebra:
\[
\mu_0: \quad  e_ie_j=e_{i+j}, 2\leq i+j \leq n,
\]
where $\{e_1, e_2, \dots , e_n\}$ is a basis of the algebra $\mathcal{A}$ and the omitted products vanish.
\end{theorem}	

{\bf Definition.} An $n$-dimensional algebra $\mathcal{A}$ is called filiform if $dim(\mathcal{A}^i)=n-i$, $2\leq i\leq n$.

\begin{theorem} [{\bf \cite{{KL}}}]  \label{2.2}
For $n>3$ every $n -$ dimensional filiform associative algebra over an algebraically closed field $\mathbb{F}$ of characteristic zero is isomorphic to one of the following pairwise non - isomorphic algebras with a basis $\{e_1, e_2, \dots , e_n\}$:
\[\mu_{1,1}:\quad   e_ie_j=e_{i+j},\]
\[\mu_{1,2}: \quad  e_ie_j=e_{i+j},e_ne_n=e_{n-1},\]
\[\mu_{1,3}:  \quad e_ie_j=e_{i+j},e_1e_n=e_{n-1},\]
\[\mu_{1,4}: \quad  e_ie_j=e_{i+j},e_1e_n=e_ne_n=e_{n-1}\]
where $2\leq i+j\leq n-1$.
\end{theorem}

\section{Description of automorphisms of finite-dimensional null-filiform and filiform  associative algebras} \label{S:big}

Here we give a description of automorphisms of the associative algebras from the theorems \ref{2.1} and \ref{2.2}.

\begin{theorem}  \label{3.1}
A linear map $\varphi:\mu_0\to \mu_0 $ is an automorphism of the algebra $\mu_0$ if and only if the map $\varphi$ has the
following form:
\[
\varphi(e_1)=\sum\limits_{i=1}^na_ie_i,
\]
\[
\varphi(e_i)=\sum\limits_{j=i}^n (\sum\limits_{k_1 +k_2 +... +k_i=j} a_{k_1}\cdot a_{k_2}\cdot ... \cdot a_{k_i})e_j,
\quad   2\leq i \leq n, \quad (1)
\]
where $a_1\neq0$.
\end{theorem}

\begin{proof}
Let
\[
\varphi(e_1)=\sum\limits_{i=1}^na_ie_i.
\]
Then
\[
\varphi(e_2)=\varphi(e_1e_1)=\varphi(e_1)\varphi(e_1)
\]
\[
=\left(\sum\limits_{i=1}^na_ie_i\right)\left(\sum\limits_{i=1}^na_ie_i\right)=\left(\sum\limits_{i=1}^na_ie_i\right)^2
\]
\[
=\sum\limits_{j=2}^n \left(\sum\limits_{k_1 +k_2=j} a_{k_1}\cdot a_{k_2}\right)e_j.
\]
Also we have
\[
\varphi(e_3)=\varphi(e_1e_2)=\varphi(e_1)\varphi(e_2)=\left(\sum\limits_{i=1}^na_ie_i\right)\left(\sum\limits_{i=1}^na_ie_i\right)^2]
\]
\[
=\left(\sum\limits_{i=1}^na_ie_i\right)^3=\sum\limits_{j=3}^n \left(\sum\limits_{k_1 +k_2+k_3=j} a_{k_1}\cdot a_{k_2}\cdot a_{k_3}\right)e_j.
\]
Similarly, for any $i=2,3,\dots n$, we have
\[
\varphi(e_i)=\varphi(e_1e_{i-1})=\varphi(e_1)\varphi(e_{i-1})=\left(\sum\limits_{i=1}^na_ie_i\right)\left(\sum\limits_{i=1}^na_ie_i\right)^{i-1}
\]
\[
=\left(\sum\limits_{i=1}^na_ie_i\right)^i=\sum\limits_{j=i}^n \left(\sum\limits_{k_1 +k_2 +... +k_i=j} a_{k_1}\cdot a_{k_2}\cdot ... \cdot a_{k_i}\right)e_j.
\]
The proof is complete.
\end{proof}

\begin{theorem}  \label{3.2}
A linear map $\varphi:\mu_{1,1} \to \mu_{1,1} $ is an automorphism if and only if the map $\varphi$ has the
following form:
\[\varphi(e_1)=\sum\limits_{i=1}^na_ie_i,\]
\[\varphi(e_i)=\sum\limits_{j=i}^{n-1} (\sum\limits_{k_1 +k_2 +... +k_i=j} a_{k_1}\cdot a_{k_2}\cdot ... \cdot a_{k_i})e_j, \quad   2\leq i \leq n-1 \]
\[\varphi(e_n)=b_{n-1}e_{n-1}+b_ne_n\]
where $a_1\neq0$.
\end{theorem}

\begin{proof}
Let
\[
\varphi(e_1)=\sum\limits_{i=1}^na_ie_i,
\]
\[
\varphi(e_n)=\sum\limits_{i=1}^nb_ie_i.
\]
Then, by the table of multiplication of the algebra $\mu_{1,1}$ and equalities
\[
\varphi(e_i)=\varphi(e_1e_{i-1}), \quad 2\leq i \leq n-1,
\]
we have
\[
\varphi(e_2)=\varphi(e_1e_1)=\varphi(e_1)\varphi(e_1)=\left(\sum\limits_{i=1}^na_ie_i\right)\left(\sum\limits_{i=1}^na_ie_i\right)
=\left(\sum\limits_{i=1}^na_ie_i\right)^2=
\]
\[
=\sum\limits_{j=2}^{n-1} \left(\sum\limits_{k_1 +k_2=j} a_{k_1}\cdot a_{k_2}\right)e_j,
\]
\[
\varphi(e_3)=\varphi(e_1e_2)=\varphi(e_1)\varphi(e_2)=\left(\sum\limits_{i=1}^na_ie_i\right)\left(\sum\limits_{i=1}^na_ie_i\right)^2=
\]
\[
=\left(\sum\limits_{i=1}^na_ie_i\right)^3=\sum\limits_{j=3}^{n-1} \left(\sum\limits_{k_1 +k_2+k_3=j} a_{k_1}\cdot a_{k_2}\cdot a_{k_3}\right)e_j,
\]
and so on
\[
\varphi(e_i)=\varphi(e_1e_{i-1})=\varphi(e_1)\varphi(e_{i-1})=\left(\sum\limits_{i=1}^na_ie_i\right)\left(\sum\limits_{i=1}^na_ie_i\right)^{i-1}]=
\]
\[
=\left(\sum\limits_{i=1}^na_ie_i\right)^i=\sum\limits_{j=i}^{n-1}
\left(\sum\limits_{k_1 +k_2 +... +k_i=j} a_{k_1}\cdot a_{k_2}\cdot ... \cdot a_{k_i}\right)e_j, \quad   2\leq i \leq n-1.
\]
We also have
\[
0=\varphi(e_1e_n)=\varphi(e_1)\varphi(e_n)=\left(\sum\limits_{i=1}^na_ie_i\right)\left(\sum\limits_{i=1}^nb_ie_i\right)=
\]
\[
=\sum\limits_{j=1}^{n-2} \sum\limits_{i=j+1}^{n-1}a_jb_{i-1}e_i.
\]
From this it follows that
\[
b_i=0, \quad 1\leq i \leq n-2,\quad b_{n-1}\neq 0, b_n \neq 0.
\]
So,
\[
\varphi(e_n)=b_{n-1}e_{n-1}+b_ne_n.
\]
The proof is complete.
\end{proof}

We can similarly prove the following theorems.

\begin{theorem}   \label{3.3}
A linear map $\varphi:\mu_{1,2} \to \mu_{1,2} $ is an automorphism if and only if the map $\varphi$ has the
following form:
\[\varphi(e_1)=\sum\limits_{i=1}^na_ie_i,\]
\[\varphi(e_2)=\sum\limits_{j=2}^{n-1} (\sum\limits_{k_1 +k_2=j} a_{k_1}\cdot a_{k_2})e_j+a^2_ne_{n-1}\]
\[\varphi(e_i)=\sum\limits_{j=i}^{n-1} (\sum\limits_{k_1 +k_2 +... +k_i=j} a_{k_1}\cdot a_{k_2}\cdot ... \cdot a_{k_i})e_j, \quad   3\leq i \leq n-1 \]
\[\varphi(e_n)=- a_n \sqrt{a^{n-3}_1}e_{n-2}+b_{n-1}e_{n-1}+ \sqrt{a^{n-1}_1}e_n\]
where $a_1\neq0$.
\end{theorem}

\begin{theorem}  \label{3.4}
A linear map $\varphi:\mu_{1,3} \to \mu_{1,3} $ is an automorphism if and only if the map $\varphi$ has the
following form:
\[\varphi(e_1)=\sum\limits_{i=1}^na_ie_i,\]
\[\varphi(e_2)=\sum\limits_{j=2}^{n-1} (\sum\limits_{k_1 +k_2=j} a_{k_1}\cdot a_{k_2})e_j+a_1a_ne_{n-1}\]
\[\varphi(e_i)=\sum\limits_{j=i}^{n-1} (\sum\limits_{k_1 +k_2 +... +k_i=j} a_{k_1}\cdot a_{k_2}\cdot ... \cdot a_{k_i})e_j, \quad   3\leq i \leq n-1 \]
\[\varphi(e_n)=b_{n-1}e_{n-1}+a^{n-2}_1e_{n}\]
where $a_1\neq0$.
\end{theorem}

\begin{theorem}  \label{3.5}
A linear map $\varphi:\mu_{1,4} \to \mu_{1,4} $ is an automorphism if and only if the map $\varphi$ has the
following form:
\[\varphi(e_1)=\sum\limits_{i=1}^na_ie_i,\]
\[\varphi(e_2)=\sum\limits_{j=2}^{n-1} (\sum\limits_{k_1 +k_2=j} a_{k_1}\cdot a_{k_2})e_j+(a_1a_n+a^2_n)e_{n-1},\]
\[\varphi(e_i)=\sum\limits_{j=i}^{n-1} (\sum\limits_{k_1 +k_2 +... +k_i=j} a_{k_1}\cdot a_{k_2}\cdot ... \cdot a_{k_i})e_j, \quad   3\leq i \leq n-1, \]
\[\varphi(e_n)=-a_n e_{n-2}+b_{n-1}e_{n-1}+e_n\]
where $a_1=1$.
\end{theorem}

\section{Description of local automorphisms of finite-dimensional null-filiform and filiform  associative algebras}

Now we give a description of local automorphisms of the associative algebras from the theorems \ref{2.1} and \ref{2.2}.

{\bf Definition.}
Let $A$ be an algebra. A linear map $\Phi : A \to A$ is called a local automorphism, if for
any element $x \in A$ there exists an automorphism $\varphi_x : A \to A$ such that $\Phi(x) = \varphi_x(x)$.

\begin{theorem} \label{0}
A linear map $\Phi$ is a local automorphism of $\mu_0$ if and only if the matrix of $\Phi$ has
the following lower triangular form
\[
\begin{pmatrix}
b_{1,1}&0&0&\dots&0&0\\
b_{2,1}&b_{2,2}&0&\dots&0&0\\
b_{3,1}&b_{3,2}&b_{3,3}&\dots&0&0\\
\vdots&\vdots&\vdots&\ddots&\vdots&\vdots\\
b_{n-1,1}&b_{n-1,2}&b_{n-1,3}&\dots&b_{n-1,n-1}&0\\
b_{n,1}&b_{n,2}&b_{n,3}&\dots&b_{n,n-1}&b_{n,n}\\
\end{pmatrix}
\]
\end{theorem}

\begin{proof}
Let $\Phi$ be an arbitrary local automorphism on $\mu_0$.
By the definition, for any element $x\in \mu_0$, there exists an automorphism $\varphi_x$ on $\mu_0$ such that
$$
\Phi(x)=\varphi_x(x).
$$
By theorem \ref{3.1}, the automorphism $\varphi_x$ has a matrix of the following form:
\begin{tiny}
\[
A_x=\begin{pmatrix}
        a_1^x&0&\dots&0&0\\
        a_2^x&(a_1^x)^2&\dots&0&0\\
        \vdots&\vdots&\ddots&\vdots&\vdots\\
         a_{n-1}^x&\sum\limits_{k_1+k_2=n-1}a^x_{k_1}a^x_{k_2}&\dots&(a_1^x)^{n-1}&0\\
         a_n^x&\sum\limits_{k_1+k_2=n}a^x_{k_1}a^x_{k_2} &\dots&\sum\limits_{{k_1+k_2+\dots+k_{n-1}}=n-1} a^x_{k_1}a^x_{k_2}\dots a^x_{k_{n-1}}&(a^x_1)^n
\end{pmatrix}.
\]
\end{tiny}

Let $A$ be the matrix of $\Phi$ and
\[
A=
\begin{pmatrix}
b_{1,1}&b_{1,2}&b_{1,3}&\dots&b_{1,n-1}&b_{1,n}\\
b_{2,1}&b_{2,2}&b_{2,3}&\dots&b_{2,n-1}&b_{2,n}\\
b_{3,1}&b_{3,2}&b_{3,3}&\dots&b_{3,n-1}&b_{3,n}\\
\vdots&\vdots&\vdots&\ddots&\vdots&\vdots\\
b_{n-1,1}&b_{n-1,2}&b_{n-1,3}&\dots&b_{n-1,n-1}&b_{n-1,n}\\
b_{n,1}&b_{n,2}&b_{n,3}&\dots&b_{n,n-1}&b_{n,n}\\
\end{pmatrix}.
\]

Then, by choosing subsequently $x=e_1$, $x=e_2$, $\ldots, x=e_n$ and using $\Phi(x)=\varphi_x(x)$, i.e. $A\bar{x}=A_x\bar{x}$,
where $\bar{x}=(x_1,x_2,\dots x_n)^T$ is the vector corresponding to $x=x_1e_1+\dots+x_ne_n$,
we have $b_{i,j}=0$, $i<j$, and $b_{k,k}\neq 0$, $1\leq k\leq n$, which implies
\[
A=\begin{pmatrix}
b_{1,1}&0&0&\dots&0&0\\
b_{2,1}&b_{2,2}&0&\dots&0&0\\
b_{3,1}&b_{3,2}&b_{3,3}&\dots&0&0\\
\vdots&\vdots&\vdots&\ddots&\vdots&\vdots\\
b_{n-1,1}&b_{n-1,2}&b_{n-1,3}&\dots&b_{n-1,n-1}&0\\
b_{n,1}&b_{n,2}&b_{n,3}&\dots&b_{n,n-1}&b_{n,n}\\
\end{pmatrix}\]

Now we prove that the linear operator, defined by the matrix $A$ is a local automorphism.
If, for each element $x\in \mu_0$, there exists a matrix $A_x$ of the form in theorem
\ref{3.1} such that
\[
A\bar{x}=A_x\bar{x},     \eqno{(1)}
\]
then the linear operator, defined by the matrix $A$ is a local automorphism.
In other words, if, for each element $x\in \mu_0$, the system of equations
\[
\left\{
\begin{array}{l}
b_{1,1}x_1=a_1^xx_1,\\
\sum\limits_{j=1}^i b_{i,j}x_j=a_i^xx_1+\sum\limits_{j=2}^{i}\sum\limits_{k_1+k_2+\dots+k_j=i}a^x_{k_1}a^x_{k_2}\dots a^x_{k_j}x_j,\quad 2\leq i\leq n,\\
\end{array}
\right.     \eqno{(2)}
\]
obtained from (1.1), has a solution with respect to the variables
$$
a_1^x,\ a_2^x,\ \dots a_n^x,
$$
then the linear operator, defined by the matrix $A$, is a local automorphism.

Let us consider the following cases
\begin{itemize}
  \item If $x_1\neq0$ then $a_1^x=b_{1,1}$,$a_2^x=b_{2,1}+\frac{1}{x_1}(b_{2,2}-(a_1^x)^2)x_2,$ \\ $a_i^x=b_{i,1}+\frac{1}{x_1}\sum\limits_{j=2}^i(b_{i,j}-\sum\limits_{k_1+k_2+\dots+k_j=i}a^x_{k_1}a^x_{k_2}\dots a^x_{k_j})x_j$, where $3\leq i\leq n$ $(a_1^x\neq0)$.
  \item If $x_1=0$ and $x_2\neq0$ then $(a_1^x)^2=b _{2,2}$,\\ $a_{i-1}^x=\frac{1}{2a_1^x} (b_{i,2}-\sum\limits_{k_1+k_2=i}^{k_1\neq1,k_2\neq1 }a_{k_1}a_{k_2} +\frac{1}{x_2}\sum\limits_{j=3}^i(b_{i,j}-\sum\limits_{k_1+k_2+\dots+k_j=i}a^x_{k_1}a^x_{k_2}\dots a^x_{k_j})x_j)$, where $3\leq i\leq n,$  $(a_1^x\neq0)$.
  \item If $x_1=0, x_2=0$ and $x_3\neq0$ then $(a_1^x)^3=b _{3,3}$,\\ $a_{i-2}^x=\frac{1}{3(a_1^x)^2} (b_{i,3}-\sum\limits_{k_1+k_2+k_3=i}^{k_{l_1}+k_{l_2}\neq2 }a_{k_1}a_{k_2}a_{k_3} +\frac{1}{x_3}\sum\limits_{j=4}^i(b_{i,j}-$\\
      $\sum\limits_{k_1+k_2+\dots+k_j=i}a^x_{k_1}a^x_{k_2}\dots a^x_{k_j})x_j)$, where $4\leq i\leq n,$ $l_1,l_2\in\{1,2,3\},$ $(a_1^x\neq0)$.
  \item If $x_1=x_2=\dots=x_{m-1}=0$ and $x_m\neq0$ then $(a_1^x)^m=b_{m,m}$,\\
  $a_{i-m+1}^x=\frac{1}{m(a_1^x)^{m-1}}\Big(b_{i,k}-\sum\limits_{k_1+k_2+\dots+k_m=i}^{k_{l_1}+k_{l_2}+\dots+k_{l_{m-1}}\neq m-1 }a_{k_1}a_{k_2}\dots a_{k_m}+\\
  +\frac{1}{x_k}\sum\limits_{j=k+1}^i\big(b_{i,j}-\sum\limits_{k_1+k_2+\dots+k_j=i}a^x_{k_1}a^x_{k_2}\dots a^x_{k_j})x_j\Big),$
  where \\ $m\leq i\leq n,$ $l_1,l_2,\dots, l_{m-1}\in\{1,2,\dots,m\},$ $(a_1^x\neq0)$.
\end{itemize}
Hence, the system of equation (2) always has a solution.
Therefore, the linear operator, defined by the matrix $A$ is
a local automorphism. This completes the proof.
\end{proof}

\begin{theorem} \label{1.1}
A linear map $\Phi$ is a local automorphism of $\mu_{1,1}$ if and only if the matrix of $\Phi$ has
the following lower triangular form
\[
\begin{pmatrix}
b_{1,1}&0&0&\dots&0&0\\
b_{2,1}&b_{2,2}&0&\dots&0&0\\
b_{3,1}&b_{3,2}&b_{3,3}&\dots&0&0\\
\vdots&\vdots&\vdots&\ddots&\vdots&\vdots\\
b_{n-1,1}&b_{n-1,2}&b_{n-1,3}&\dots&b_{n-1,n-1}&b_{n-1,n}\\
b_{n,1}& 0 & 0 &\dots& 0 &b_{n,n}\\
\end{pmatrix}
\]
\end{theorem}

\begin{proof}
Let $\Phi$ be an arbitrary local automorphism on $\mu_{1,1}$.
By the definition, for any element $x\in \mu_{1,1}$, there exists an automorphism $\varphi_x$ on $\mu_{1,1}$ such that
$$
\Phi(x)=\varphi_x(x).
$$
By theorem \ref{3.2}, the automorphism $\varphi_x$ has a matrix of the following form:
\begin{tiny}
\[
A_x=\begin{pmatrix}
        a_1^x&0&0&\dots&0&0\\
        a_2^x&(a_1^x)^2&0&\dots&0&0\\
        a_3^x&\sum\limits_{k_1+k_2=3}a^x_{k_1}a^x_{k_2}&(a_1^x)^3&\dots&0&0\\
        \vdots&\vdots&\vdots&\ddots&\vdots&\vdots\\
        a_{n-1}^x&\sum\limits_{k_1+k_2=n-1}a^x_{k_1}a^x_{k_2}&\sum\limits_{k_1+k_2+k_3=n-1}a^x_{k_1}a^x_{k_2}a^x_{k_3} &\dots&(a_1^x)^{n-1}&b^x_{n-1}\\
        a_n^x&0&0&\dots&0&b^x_n
\end{pmatrix}.
\]
\end{tiny}

Let $A$ be the matrix of $\Phi$ and
\[
A=
\begin{pmatrix}
b_{1,1}&b_{1,2}&b_{1,3}&\dots&b_{1,n-1}&b_{1,n}\\
b_{2,1}&b_{2,2}&b_{2,3}&\dots&b_{2,n-1}&b_{2,n}\\
b_{3,1}&b_{3,2}&b_{3,3}&\dots&b_{3,n-1}&b_{3,n}\\
\vdots&\vdots&\vdots&\ddots&\vdots&\vdots\\
b_{n-1,1}&b_{n-1,2}&b_{n-1,3}&\dots&b_{n-1,n-1}&b_{n-1,n}\\
b_{n,1}&b_{n,2}&b_{n,3}&\dots&b_{n,n-1}&b_{n,n}\\
\end{pmatrix}.
\]

Then, by choosing subsequently $x=e_1$, $x=e_2$, $\ldots, x=e_n$ and using $\Phi(x)=\varphi_x(x)$, i.e. $A\bar{x}=A_x\bar{x}$,
where $\bar{x}=(x_1,x_2,\dots x_n)^T$ is the vector corresponding to $x=x_1e_1+\dots+x_ne_n$,
we have $b_{i,j}=0$, $i<j$, $i\neq n-1$, and $b_{k,k}\neq 0$, $1\leq k\leq n$, which implies
\[
A=\begin{pmatrix}
b_{1,1}&0&0&\dots&0&0\\
b_{2,1}&b_{2,2}&0&\dots&0&0\\
b_{3,1}&b_{3,2}&b_{3,3}&\dots&0&0\\
\vdots&\vdots&\vdots&\ddots&\vdots&\vdots\\
b_{n-1,1}&b_{n-1,2}&b_{n-1,3}&\dots&b_{n-1,n-1}&b_{n-1,n}\\
b_{n,1}&0&0&\dots&0&b_{n,n}\\
\end{pmatrix}\]

Similar to the proof of Theorem \ref{0} we prove that, for each element $x\in \mu_{1,1}$, the system of equations
\[
\left\{
\begin{array}{l}
b_{1,1}x_1=a_1^xx_1,\\
\sum\limits_{j=1}^i b_{i,j}x_j=a_i^xx_1+\sum\limits_{j=2}^{i}\sum\limits_{k_1+k_2+\dots+k_j=i}a^x_{k_1}a^x_{k_2}\dots a^x_{k_j}x_j,\quad 2\leq i\leq n-2,\\
\sum\limits_{j=1}^{n} b_{n-1,j}x_j=a_{n-1}^xx_1+\sum\limits_{j=2}^{n-1}\sum\limits_{k_1+k_2+\dots+k_j=n-1}a^x_{k_1}a^x_{k_2}\dots a^x_{k_j}x_j+b_{n-1}^xx_n,\\
b_{n,1}x_1+b_{n,n}x_n=a_n^xx_1+b_n^xx_n,
\end{array}
\right.           \eqno{(3)}
\]
obtained from the equality $A\bar{x}=A_x\bar{x}$, has a solution with respect to the variables
$$
a_1^x,\ a_2^x,\ \dots a_n^x, b_{n-1}^x, b_n^x.
$$

Let us consider the following cases
\begin{itemize}

    \item If $x_1\neq0$ then $a_1^x=b_{1,1}$,\\ $a_i^x=b_{i,1}+\frac{1}{x_1}\sum\limits_{j=2}^i(b_{i,j}-\sum\limits_{k_1+k_2+\dots+k_j=i}a^x_{k_1}a^x_{k_2}\dots a^x_{k_j})x_j$, where $2\leq i\leq n-2$,\\
        $a_{n-1}^x=b_{n-1,n}+\frac{1}{x_1}\sum\limits_{j=2}^{n-1}(b_{n-1,j}-\sum\limits_{k_1+k_2+\dots+k_j=n-1}a^x_{k_1}a^x_{k_2}\dots a^x_{k_j})x_j+\frac{1}{x_1}(b_{n,n}-b_n^x)x_n$, \\
        $a_n^x=b_{n,1}+\frac{1}{x_1}(b_{n,n}-b_n^x)x_n$,\\
        where $b_{n-1}^x$ and $b_n^x$ are defined arbitrarily.
  \item If $x_1=0$ and $x_2\neq0$ then $(a_1^x)^2=b _{2,2}$,\\ $a_{i-1}^x=\frac{1}{2a_1^x} (b_{i,2}-\sum\limits_{k_1+k_2=i}^{k_1\neq1,k_2\neq1 }a^x_{k_1}a^x_{k_2} +\frac{1}{x_2}\sum\limits_{j=3}^i(b_{i,j}-\sum\limits_{k_1+k_2+\dots+k_j=i}a^x_{k_1}a^x_{k_2}\dots a^x_{k_j})x_j)$, where $3\leq i\leq n-2,$  $(a_1^x\neq0)$,\\
      $a_{n-2}^x=\frac{1}{2a_1^x} (b_{n-1,2}-\sum\limits_{k_1+k_2=n-1}^{k_1\neq1,k_2\neq1 }a^x_{k_1}a^x_{k_2} +\frac{1}{x_2}(\sum\limits_{j=3}^{n-1}(b_{n-1,j}-$\\
      $\sum\limits_{k_1+k_2+\dots+k_j=n-1}a^x_{k_1}a^x_{k_2}\dots a^x_{k_j})x_j+(b_{n-1,n}-b_{n-1}^x)x_n))$, \\
      $b_n^x=b_{n,n}$ where $a_{n}^x$, $a_{n-1}^x$ and $ b_{n-1}^x$ are defined arbitrarily.
  \item If $x_1=0, x_2=0$ and $x_3\neq0$ then $(a_1^x)^3=b _{3,3}$,\\ $a_{i-2}^x=\frac{1}{3(a_1^x)^2} (b_{i,3}-\sum\limits_{k_1+k_2+k_3=i}^{k_{l_1}+k_{l_2}\neq2 }a_{k_1}a_{k_2}a_{k_3} +\frac{1}{x_3}\sum\limits_{j=4}^i(b_{i,j}-$\\
      $\sum\limits_{k_1+k_2+\dots+k_j=i}a^x_{k_1}a^x_{k_2}\dots a^x_{k_j})x_j)$, where $4\leq i\leq n-2,$ $l_1,l_2\in\{1,2,3\},$ $(a_1^x\neq0)$,\\
      $a_{n-3}^x=\frac{1}{3(a_1^x)^2} (b_{n-1,3}-\sum\limits_{k_1+k_2+k_3=n-1}^{k_{l_1}+k_{l_2}\neq2 }a^x_{k_1}a^x_{k_2}a^x_{k_3}$\\ $+\frac{1}{x_3}(\sum\limits_{j=4}^{n-1}(b_{n-1,j}-\sum\limits_{k_1+k_2+\dots+k_j=n-1}a^x_{k_1}a^x_{k_2}\dots a^x_{k_j})x_j+(b_{n-1,n}-b_{n-1}^x)x_n))$ \\ where $l_1,l_2\in\{1,2,3\},$ $(a_1^x\neq0)$\\
      $b_n^x=b_{n,n}$ where $a_{n}^x$, $a_{n-1}^x$, $a_{n-2}^x$ and $ b_{n-1}^x$ are defined arbitrarily.
  \item If $x_1=x_2=\dots=x_{m-1}=0$ and $x_m\neq0$ then $(a_1^x)^m=b_{m,m}$,\\
  $a_{i-m+1}^x=\frac{1}{m(a_1^x)^{m-1}}\Big(b_{i,m}-\sum\limits_{k_1+k_2+\dots+k_m=i}^{k_{l_1}+k_{l_2}+\dots+k_{l_{m-1}}\neq m-1 }a^x_{k_1}a^x_{k_2}\dots a^x_{k_m}+\\
  +\frac{1}{x_k}\sum\limits_{j=m+1}^i\big(b_{i,j}-\sum\limits_{k_1+k_2+\dots+k_j=i}a^x_{k_1}a^x_{k_2}\dots a^x_{k_j})x_j\Big),$
  where $m\leq i\leq n-2,$ $l_1,l_2,\dots, l_{m-1}\in\{1,2,\dots,m\},$ $(a_1^x\neq0)$,\\
  $a_{n-m}^x=\frac{1}{m(a_1^x)^{m-1}}\Big(b_{n-1,k}-\sum\limits_{k_1+k_2+\dots+k_m=n-1}^{k_{l_1}+k_{l_2}+\dots+k_{l_{m-1}}\neq m-1 }a^x_{k_1}a^x_{k_2}\dots a^x_{k_m}+\\
  +\frac{1}{x_k}\sum\limits_{j=m+1}^{n-1}\big(b_{n-1,j}-\sum\limits_{k_1+k_2+\dots+k_j=n-1}a^x_{k_1}a^x_{k_2}\dots a^x_{k_j})x_j\Big),$
  \\ where $l_1,l_2,\dots, l_{m-1}\in\{1,2,\dots,m\},$ $(a_1^x\neq0)$.\\
  $b_n^x=b_{n,n}$ where $a_{n}^x$, $a_{n-1}^x$,$\dots $ , $a_{n-m+1}^x$ and $ b_{n-1}^x$ are defined arbitrarily.
\end{itemize}
Hence, the system of equation (3) always has a solution.
Therefore, the linear operator, defined by the matrix $A$ is
a local automorphism. The proof is complete.
\end{proof}

\begin{theorem} \label{1.2}
A linear map $\Phi$ is a local automorphism of $\mu_{1,2}$ if and only if the matrix of $\Phi$ has
the following lower triangular form
\[
\begin{pmatrix}
b_{1,1}&0&0&\dots&0&0\\
b_{2,1}&b_{2,2}&0&\dots&0&0\\
b_{3,1}&b_{3,2}&b_{3,3}&\dots&0&0\\
\vdots&\vdots&\vdots&\ddots&\vdots&\vdots\\
b_{n-1,1}&b_{n-1,2}&b_{n-1,3}&\dots&b_{n-1,n-1}&b_{n-1,n}\\
b_{n,1}& 0 & 0 &\dots& 0 &b_{n,n}\\
\end{pmatrix}
\]
\end{theorem}

\begin{proof}
Let $\Phi$ be an arbitrary local automorphism on $\mu_{1,2}$.
By the definition, for any element $x\in \mu_{1,2}$, there exists an automorphism $\varphi_x$ on $\mu_{1,2}$ such that
$$
\Phi(x)=\varphi_x(x).
$$
By theorem \ref{3.3}, the automorphism $\varphi_x$ has a matrix of the following form:
\begin{tiny}
\[
A_x=\begin{pmatrix}
        a_1^x&0&0&\dots&0&0\\
        a_2^x&(a_1^x)^2&0&\dots&0&0\\
        a_3^x&\sum\limits_{k_1+k_2=3}a^x_{k_1}a^x_{k_2}&(a_1^x)^3&\dots&0&0\\
        \vdots&\vdots&\vdots&\ddots&\vdots&\vdots\\
        a_{n-1}^x&\sum\limits_{k_1+k_2=n-1}a^x_{k_1}a^x_{k_2}+(a_n^x)^2&\sum\limits_{k_1+k_2+k_3=n-1}a^x_{k_1}a^x_{k_2}a^x_{k_3} &\dots&(a_1^x)^{n-1}&b^x_{n-1}\\
        a_n^x&0&0&\dots&0&\sqrt{(a_1^x)^{n-1}}
\end{pmatrix}.
\]
\end{tiny}

Let $A$ be the matrix of $\Phi$ and
\[
A=
\begin{pmatrix}
b_{1,1}&b_{1,2}&b_{1,3}&\dots&b_{1,n-1}&b_{1,n}\\
b_{2,1}&b_{2,2}&b_{2,3}&\dots&b_{2,n-1}&b_{2,n}\\
b_{3,1}&b_{3,2}&b_{3,3}&\dots&b_{3,n-1}&b_{3,n}\\
\vdots&\vdots&\vdots&\ddots&\vdots&\vdots\\
b_{n-1,1}&b_{n-1,2}&b_{n-1,3}&\dots&b_{n-1,n-1}&b_{n-1,n}\\
b_{n,1}&b_{n,2}&b_{n,3}&\dots&b_{n,n-1}&b_{n,n}\\
\end{pmatrix}.
\]

Then, by choosing subsequently $x=e_1$, $x=e_2$, $\ldots, x=e_n$ and using $\Phi(x)=\varphi_x(x)$, i.e. $A\bar{x}=A_x\bar{x}$,
where $\bar{x}=(x_1,x_2,\dots x_n)^T$ is the vector corresponding to $x=x_1e_1+\dots+x_ne_n$,
we have $b_{i,j}=0$, $i<j$, $i\neq n-1$, and $b_{k,k}\neq 0$, $1\leq k\leq n$, which implies
\[
A=\begin{pmatrix}
b_{1,1}&0&0&\dots&0&0\\
b_{2,1}&b_{2,2}&0&\dots&0&0\\
b_{3,1}&b_{3,2}&b_{3,3}&\dots&0&0\\
\vdots&\vdots&\vdots&\ddots&\vdots&\vdots\\
b_{n-1,1}&b_{n-1,2}&b_{n-1,3}&\dots&b_{n-1,n-1}&b_{n-1,n}\\
b_{n,1}&0&0&\dots&0&b_{n,n}\\
\end{pmatrix}\]

Similar to the proof of Theorem \ref{0} we prove that, for each element $x\in \mu_{1,2}$, the system of equations
\[
\left\{
\begin{array}{l}
b_{1,1}x_1=a_1^xx_1,\\
\sum\limits_{j=1}^i b_{i,j}x_j=a_i^xx_1+\sum\limits_{j=2}^{i}\sum\limits_{k_1+k_2+\dots+k_j=i}a^x_{k_1}a^x_{k_2}\dots a^x_{k_j}x_j,\quad 2\leq i\leq n-3,\\
\sum\limits_{j=1}^{n-2} b_{n-2,j}x_j+b_{n-2,n}x_n=\\ =a_{n-2}^xx_1+\sum\limits_{j=2}^{n-2}\sum\limits_{k_1+k_2+\dots+k_j=n-2}a^x_{k_1}a^x_{k_2}\dots a^x_{k_j}x_j-a_{n}^x\sqrt{(a^x_1)^{n-3}}x_n,\\
\sum\limits_{j=1}^{n} b_{n-1,j}x_j=a_{n-1}^xx_1+\sum\limits_{j=2}^{n-1}\sum\limits_{k_1+k_2+\dots+k_j=n-1}a^x_{k_1}a^x_{k_2}\dots a^x_{k_j}x_j+(a^x_n)^2x_2+b_{n-1}^xx_n,\\
b_{n,1}x_1+b_{n,n}x_n=a_n^xx_1+\sqrt{(a^x_1)^{n-1}}x_n,
\end{array}
\right.                     \eqno{(4)}
\]
obtained from the equality $A\bar{x}=A_x\bar{x}$, has a solution with respect to the variables
$$
a_1^x,\ a_2^x,\ \dots a_n^x, b_{n-1}^x.
$$

Let us consider the following cases
\begin{itemize}

    \item If $x_1\neq0$ then $a_1^x=b_{1,1}$,\\ $a_i^x=b_{i,1}+\frac{1}{x_1}\sum\limits_{j=2}^i(b_{i,j}-\sum\limits_{k_1+k_2+\dots+k_j=i}a^x_{k_1}a^x_{k_2}\dots a^x_{k_j})x_j$, where $2\leq i\leq n-3$,\\
        $a_{n-2}^x=b_{n-2,1}+\frac{1}{x_1}\sum\limits_{j=2}^{n-2}(b_{n-2,j}-\sum\limits_{k_1+k_2+\dots+k_j=n-2}a^x_{k_1}a^x_{k_2}\dots a^x_{k_j})x_j+\frac{1}{x_1}(b_{n-2,n}-a_n^x\sqrt{(a^x_1)^{n-3}})x_n$,where $a_n^x$ is defined arbitrarily, \\
        $a_{n-1}^x=b_{n-1,1}+\frac{1}{x_1}\sum\limits_{j=2}^{n-1}(b_{n-1,j}-\sum\limits_{k_1+k_2+\dots+k_j=n-1}a^x_{k_1}a^x_{k_2}\dots a^x_{k_j})x_j+\frac{1}{x_1}((b_{n-1,n}-b_{n-1}^x)x_n-(a_n^x)^2x_2)$,where $a_n^x$, $b_{n-1}^x$ are defined arbitrarily. \\
        $a_n^x=b_{n,1}+\frac{1}{x_1}(b_{n,n}-\sqrt{(a_1^x)^{n-1}})x_n$.\\

  \item If $x_1=0$ and $x_2\neq0$ then $(a_1^x)^2=b _{2,2}$,\\ $a_{i-1}^x=\frac{1}{2a_1^x} (b_{i,2}-\sum\limits_{k_1+k_2=i}^{k_1\neq1,k_2\neq1 }a^x_{k_1}a^x_{k_2} +\frac{1}{x_2}\sum\limits_{j=3}^i(b_{i,j}-\sum\limits_{k_1+k_2+\dots+k_j=i}a^x_{k_1}a^x_{k_2}\dots a^x_{k_j})x_j)$, where $3\leq i\leq n-3,$  $(a_1^x\neq0)$, \\
      $a_{n-3}^x=\frac{1}{2a_1^x} (b_{n-2,2}-\sum\limits_{k_1+k_2=n-2}^{k_1\neq1,k_2\neq1 }a^x_{k_1}a^x_{k_2} +\frac{1}{x_2}(\sum\limits_{j=3}^{n-2}(b_{n-2,j}-$\\$\sum\limits_{k_1+k_2+\dots+k_j=n-2}a^x_{k_1}a^x_{k_2}\dots a^x_{k_j})x_j+(b_{n-2,n}+a_n^x \sqrt{(a_1^x)^{n-3}})x_n))$, \\
      $a_{n-2}^x=\frac{1}{2a_1^x} (b_{n-1,2}-\sum\limits_{k_1+k_2=n-1}^{k_1\neq1,k_2\neq1 }a^x_{k_1}a^x_{k_2}-a_n^x +\frac{1}{x_2}(\sum\limits_{j=3}^{n-1}(b_{n-1,j}-$\\$\sum\limits_{k_1+k_2+\dots+k_j=n-1}a^x_{k_1}a^x_{k_2}\dots a^x_{k_j})x_j+(b_{n-1,n}-b_{n-1}^x)x_n ))$, \\
      $a^x_1=(b_{n,n})^{\frac{2}{n-1}}$.

  \item If $x_1=0, x_2=0$ and $x_3\neq0$ then $(a_1^x)^3=b _{3,3}$,\\ $a_{i-2}^x=\frac{1}{3(a_1^x)^2} (b_{i,3}-\sum\limits_{k_1+k_2+k_3=i}^{k_{l_1}+k_{l_2}\neq2 }a^x_{k_1}a^x_{k_2}a^x_{k_3} +\frac{1}{x_3}\sum\limits_{j=4}^i(b_{i,j}-$\\$\sum\limits_{k_1+k_2+\dots+k_j=i}a^x_{k_1}a^x_{k_2}\dots a^x_{k_j})x_j)$, where $4\leq i\leq n-3,$ $l_1,l_2\in\{1,2,3\},$ $(a_1^x\neq0)$,\\
      $a_{n-4}^x=\frac{1}{3(a_1^x)^2} (b_{n-2,3}-\sum\limits_{k_1+k_2+k_3=n-2}^{k_{l_1}+k_{l_2}\neq 2}a^x_{k_1}a^x_{k_2}a^x_{k_3}$\\ $+\frac{1}{x_3}(\sum\limits_{j=4}^{n-2}(b_{n-2,j}-\sum\limits_{k_1+k_2+\dots+k_j=n-2}a^x_{k_1}a^x_{k_2}\dots a^x_{k_j})x_j$\\
      $+(b_{n-2,n}+a_n^x \sqrt{(a_1^x)^{n-3}})x_n))$ where $l_1,l_2\in\{1,2,3\},$ $(a_1^x\neq0)$\\

  $a_{n-3}^x=\frac{1}{3(a_1^x)^2} (b_{n-1,3}-\sum\limits_{k_1+k_2+k_3=n-1}^{k_{l_1}+k_{l_2}\neq2 }a^x_{k_1}a^x_{k_2}a^x_{k_3}$\\ $+\frac{1}{x_3}(\sum\limits_{j=4}^{n-1}(b_{n-1,j}-\sum\limits_{k_1+k_2+\dots+k_j=n-1}a^x_{k_1}a^x_{k_2}\dots a^x_{k_j})x_j$\\
  $+(b_{n-1,n}-b_{n-1}^x)x_n))$ where $l_1,l_2\in\{1,2,3\},$ $(a_1^x\neq0)$\\
      $b_n^x=b_{n,n}$ where $a_{n}^x$, $a_{n-1}^x$, $a_{n-2}^x$ and $ b_{n-1}^x$ are defined arbitrarily.

  \item If $x_1=x_2=\dots=x_{m-1}=0$ and $x_m\neq0$ then $(a_1^x)^m=b_{m,m}$,\\
  $a_{i-m+1}^x=\frac{1}{m(a_1^x)^{m-1}}\Big(b_{i,m}-\sum\limits_{k_1+k_2+\dots+k_m=i}^{k_{l_1}+k_{l_2}+\dots+k_{l_{m-1}}\neq m-1 }a^x_{k_1}a^x_{k_2}\dots a^x_{k_m}+\\
  +\frac{1}{x_m}\sum\limits_{j=m+1}^i\big(b_{i,j}-\sum\limits_{k_1+k_2+\dots+k_j=i}a^x_{k_1}a^x_{k_2}\dots a^x_{k_j})x_j\Big),$
  \\ where $m+1\leq i\leq n-3,$ $l_1,l_2,\dots, l_{m-1}\in\{1,2,\dots,m\},$ $(a_1^x\neq0)$,\\
  $a_{n-m-1}^x=\frac{1}{m(a_1^x)^{m-1}}\Big(b_{n-2,m}-\sum\limits_{k_1+k_2+\dots+k_m=n-2}^{k_{l_1}+k_{l_2}+\dots+k_{l_{m-1}}\neq m-1 }a^x_{k_1}a^x_{k_2}\dots a^x_{k_m}+\\
  +\frac{1}{x_m}((\sum\limits_{j=m+1}^{n-2}\big(b_{n-2,j}-\sum\limits_{k_1+k_2+\dots+k_j=n-2}a^x_{k_1}a^x_{k_2}\dots a^x_{k_j})x_j+(b_{n-2,n}+$\\$a_n^x \sqrt{(a_1^x)^{n-3}})x_n)\Big),$
  where $l_1,l_2,\dots, l_{m-1}\in\{1,2,\dots,m\},$ $(a_1^x\neq0)$.\\

  $a_{n-m}^x=\frac{1}{m(a_1^x)^{m-1}}\Big(b_{n-1,m}-\sum\limits_{k_1+k_2+\dots+k_m=n-1}^{k_{l_1}+k_{l_2}+\dots+k_{l_{m-1}}\neq m-1 }a^x_{k_1}a^x_{k_2}\dots a^x_{k_m}+\\
  +\frac{1}{x_m}((\sum\limits_{j=m+1}^{n-1}\big(b_{n-1,j}-\sum\limits_{k_1+k_2+\dots+k_j=n-1}a^x_{k_1}a^x_{k_2}\dots a^x_{k_j})x_j+(b_{n-1,n}-b_{n-1}^x)x_n)\Big),$
  $b_n^x=b_{n,n}$ where $a_{n}^x$, $a_{n-1}^x$,$\dots $ , $a_{n-m+1}^x$ and $ b_{n-1}^x$ are defined arbitrarily.
\end{itemize}
Hence, the system of equation (4) always has a solution.
Therefore, the linear operator, defined by the matrix $A$ is
a local automorphism. This ends the proof.
\end{proof}

We can similarly prove the following theorems using theorems \ref{3.4} and \ref{3.5}.

\begin{theorem} \label{1.3}
A linear map $\Phi$ is a local automorphism of $\mu_{1,3}$ if and only if the matrix of $\Phi$ has
the following lower triangular form
\[
\begin{pmatrix}
b_{1,1}&0&0&\dots&0&0\\
b_{2,1}&b_{2,2}&0&\dots&0&0\\
b_{3,1}&b_{3,2}&b_{3,3}&\dots&0&0\\
\vdots&\vdots&\vdots&\ddots&\vdots&\vdots\\
b_{n-1,1}&b_{n-1,2}&b_{n-1,3}&\dots&b_{n-1,n-1}&b_{n-1,n}\\
b_{n,1}& 0 & 0 &\dots& 0 &b_{n,n}\\
\end{pmatrix}
\]
\end{theorem}

\begin{theorem} \label{1.4}
A linear map $\Phi$ is a local automorphism of $\mu_{1,4}$ if and only if the matrix of $\Phi$ has
the following lower triangular form
\[
\begin{pmatrix}
b_{1,1}&0&0&\dots&0&0\\
b_{2,1}&b_{2,2}&0&\dots&0&0\\
b_{3,1}&b_{3,2}&b_{3,3}&\dots&0&0\\
\vdots&\vdots&\vdots&\ddots&\vdots&\vdots\\
b_{n-1,1}&b_{n-1,2}&b_{n-1,3}&\dots&b_{n-1,n-1}&b_{n-1,n}\\
b_{n,1}& 0 & 0 &\dots& 0 & 1\\
\end{pmatrix}
\]
\end{theorem}

\begin{remark}
Note that the common form of the matrix of a local automorphism
on an algebra includes the common form of the matrix of an automorphism on this algebra.
The coincidence of these common forms denotes that every local
automorphism of the considering algebra is an automorphism.
But the common form of the matrix
of an automorphism on the associative algebras $\mu_{0}$, $\mu_{1,1}$, $\mu_{1,2}$, $\mu_{1,3}$ and $\mu_{1,4}$
does not coincide with the common form of the matrix of a local automorphism on these algebras by
the appropriate theorems and theorems. Therefore, the associative algebras
$\mu_{0}$, $\mu_{1,1}$, $\mu_{1,2}$, $\mu_{1,3}$ and $\mu_{1,4}$ have local
automorphisms that are not automorphisms.

Also, note that local automorphisms of an arbitrary low-dimension algebra can be similarly described using
a common form of the matrix of automorphisms on this algebra. A technique for constructing a local automorphism,
which is not an automorphism, developed by us, can be applied to an arbitrary low-dimension algebra,
automorphisms of which have a matrix of a common form.
\end{remark}

\section{Description of 2-local automorphisms of finite-dimensional null-filiform and filiform associative algebras}

\begin{theorem} \label{5.1}
Each 2-local automorphism of $\mu_0$ is an automorphism.
\end{theorem}

\begin{proof}
\quad Let $\phi$ be an arbitrary 2 -local automorphism of $\mu_0$. Then, by the definition, for every element $x\in \mu_0$,
\[x=\sum\limits_{i=1}^n x_i e_i,\] \\
there exist a matrix $A_{x,e_1}$
\begin{tiny}
\[
A_{x,e_1}=
\begin{pmatrix}
        a_1^{x,e_1}&0&\dots&0&0\\
        a_2^{x,e_1}&(a_1^{x,e_1})^2&\dots&0&0\\
        \vdots&\vdots&\ddots&\vdots&\vdots\\
         a_{n-1}^{x,e_1}&\sum\limits_{k_1+k_2=n-1}a^{x,e_1}_{k_1}a^{x,e_1}_{k_2}&\dots&(a_1^{x,e_1})^{n-1}&0\\
         a_n^{x,e_1}&\sum\limits_{k_1+k_2=n}a^{x,e_1}_{k_1}a^{x,e_1}_{k_2}&\dots&\sum\limits_{{k_1+k_2+\dots+k_{n-1}}=n-1} a^{x,e_1}_{k_1}a^{x,e_1}_{k_2}\dots a^{x,e_1}_{k_{n-1}}&(a^{x,e_1}_1)^n
\end{pmatrix}
,\]
\end{tiny}

such that $\phi(x)=\widehat{A_{x,e_1} \bar{x}}$, where $\bar{x} = (x_1, x_2,\dots, x_n)^T$ is the vector corresponding
to $x$, $\widehat{\bar{x}}$ is an operation on $\bar{x}$ such that $\widehat{\bar{x}}=x$, and
\[
\phi(e_1)=\widehat{A_{x,e_1}\overline{e_1}}=
\widehat{(a_1^{x,e_1},a_2^{x,e_1},a_3^{x,e_1},\dots,a_1^{x,e_1})^T}.
\]

Since $\phi(e_1) = \varphi_{x,e_1}(e_1) = \varphi_{y,e_1}(e_1)$, we have
\[
\phi(e_1)=\widehat{(a_1^{x,e_1},a_2^{x,e_1},a_3^{x,e_1},\dots,a_n^{x,e_1})^T}=
\]
\[
=\widehat{(a_1^{y,e_1},a_2^{y,e_1},a_3^{y,e_1},\dots,a_n^{y,e_1})^T}
\]
for each pair, $x$, $y$ of elements in $\mu_0$. Hence, $a_k^{x,e_1}=a_k^{y,e_1}$, $k=1,2,\dots n$. Therefore
\[
\phi(x)=\widehat{A_{y,e_1}\bar{x}}
\] \\
for any $x\in \mu_0$, and the matrix of $\phi(x)$ does not depend on $x$.
Hence $\phi$ is a linear operator, and the matrix of $\varphi_{y,e_1}$ is the matrix of $\phi$.
Thus, by Proposition \ref{3.1}, $\phi$ is an automorphism.
\end{proof}

\begin{theorem} \label{5.2}
Each 2-local automorphism of $\mu_{1,1}$ is an automorphism.
\end{theorem}

\begin{proof}
\quad Let $\phi$ be an arbitrary 2 -local automorphism of $\mu_{1,1}$. Then, by the definition, for every element $x\in \mu_{1,1}$,
\[x=\sum\limits_{i=1}^n x_i e_i,\] \\
there exist a matrix $A_{x,e_1}$
\begin{tiny}
\[
A_{x,e_1}=
\begin{pmatrix}
        a_1^{x,e_1}&0&0&\dots&0&0\\
        a_2^{x,e_1}&(a_1^{x,e_1})^2&0&\dots&0&0\\
        a_3^{x,e_1}&\sum\limits_{k_1+k_2=3}a^{x,e_1}_{k_1}a^{x,e_1}_{k_2}&(a_1^{x,e_1})^3&\dots&0&0\\
        \vdots&\vdots&\vdots&\ddots&\vdots&\vdots\\
         a_{n-1}^{x,e_1}&\sum\limits_{k_1+k_2=n-1}a^{x,e_1}_{k_1}a^{x,e_1}_{k_2}&\sum\limits_{k_1+k_2+k_3=n-1}a^{x,e_1}_{k_1}a^{x,e_1}_{k_2}a^{x,e_1}_{k_3} &\dots&(a_1^{x,e_1})^{n-1}&b^{x,e_1}_{n-1}\\
         a_n^{x,e_1}&0&0&\dots&0&b^{x,e_1}_n
\end{pmatrix},\]
\end{tiny}

such that $\phi(x)=\widehat{A_{x,e_1} \bar{x}}$, where $\bar{x} = (x_1, x_2,\dots, x_n)^T$ is the vector corresponding
to $x$, $\widehat{\bar{x}}$ is an operation on $\bar{x}$ such that $\widehat{\bar{x}}=x$, and
\[
\phi(e_1)=\widehat{A_{x,e_1}\overline{e_1}}=
\widehat{(a_1^{x,e_1},a_2^{x,e_1},a_3^{x,e_1},\dots,a_n^{x,e_1})^T}.
\]

Since $\phi(e_1) = \varphi_{x,e_1}(e_1) = \varphi_{y_1,e_1}(e_1)$, we have
\[
\phi(e_1)=\widehat{(a_1^{x,e_1},a_2^{x,e_1},a_3^{x,e_1},\dots,a_n^{x,e_1})^T}=
\]
\[
=\widehat{(a_1^{y_1,e_1},a_2^{y_1,e_1},a_3^{y_1,e_1},\dots,a_n^{y_1,e_1})^T}
\]
for each pair, $x$, $y_1$ of elements in $\mu_{1,1}$. Hence, $a_k^{x,e_1}=a_k^{y_1,e_1}$, $k=1,2,\dots n$.

Similarly, from $\varphi_{e_n,x}(e_n)=\varphi_{e_n,y_2}(e_n)$ it follows that
\[
b^{e_n,x}_{n-1}=b^{e_n,y_2}_{n-1}, b^{e_n,x}_{n}=b^{e_n,y_2}_{n}.
\]

Therefore, if we take $y_1=e_n$, $y_2=e_1$, then
\[
\phi(x)=\varphi_{x,e_1}(x)=\varphi_{e_n,x}(x)=\varphi_{e_1,e_n}(x)
\]
for any $x\in \mu_{1,1}$, and the matrix of $\phi(x)$ does not depend on $x$.
Hence $\phi$ is a linear operator and the matrix of $\varphi_{e_1,e_n}$ is the matrix of $\phi$.
Thus, by Proposition \ref{3.2}, $\phi$ is an automorphism.
\end{proof}

The following theorem is proved similar to the proof of Theorem \ref{5.2} using theorems \ref{3.3}, \ref{3.4} and \ref{3.5}.

\begin{theorem} \label{5.3}
Each 2-local automorphism of the algebras $\mu_{1,2}$, $\mu_{1,3}$ and $\mu_{1,4}$ is an automorphism.
\end{theorem}

{\bf Conflicts of Interest:}  \ {\it On behalf of all authors, the corresponding author states that there is no conflict of interest.}

\end{document}